\documentclass[reqno,twoside]{amsart} 

\usepackage{amsmath,amsthm,amssymb,amsfonts,mathrsfs,color}
\usepackage{latexsym,esint}

\definecolor{purple}{rgb}{0.8,0.01,0.7}

\theoremstyle{plain}
\begingroup
\newtheorem{thm}{Theorem}[section]
\newtheorem{lem}[thm]{Lemma}
\newtheorem{prop}[thm]{Proposition}

\endgroup

\theoremstyle{definition}
\begingroup
\newtheorem{defn}[thm]{Definition}
\newtheorem{example}[thm]{Example}
\newtheorem{rem}[thm]{Remark}
\endgroup

\numberwithin{equation}{section}

\makeatletter
\def\rightharpoonupfill@{\arrowfill@\relbar\relbar\rightharpoonup}
\newcommand{\xrightharpoonup}[2][]{\ext@arrow 0359\rightharpoonupfill@{#1}{#2}}
\makeatother

\newcommand{\N}{\mathbb N} 
\newcommand{\R}{\mathbb R} 
\newcommand{\CC}{\mathbb C} 
\newcommand{\Ms}{{\mathbb M}^{n{\times}n}_{sym}}
\newcommand{\MD}{{\mathbb M}^{n{\times}n}_D}

\renewcommand{\div}{\operatorname{div}}
\newcommand{\wto}{\rightharpoonup}
\newcommand{\e}{\varepsilon}

\newcommand{\A}{{\mathcal A}}
\newcommand{\Q}{{\mathcal Q}}
\newcommand{\LL}{{\mathcal L}}
\newcommand{\HH}{{\mathcal H}}
\newcommand{\M}{{\mathcal M}}
\newcommand{\C}{{\mathcal C}}

\newcommand{\K}{{\mathcal K}}
\newcommand{\V}{{\mathcal V}}

\let\O=\Omega

\definecolor{purple}{rgb}{0.8,0.01,0.7}

\begin{document}
 
\title[A note on the derivation of rigid-plastic models]{A note on the derivation of rigid-plastic models}
\author[J.-F. Babadjian]{Jean-Fran\c cois Babadjian}
\author[G. A. Francfort]{Gilles A. Francfort}

\address[J.-F. Babadjian]{Sorbonne Universit\'es, UPMC Univ Paris 06, CNRS, UMR 7598, Laboratoire Jacques-Louis Lions, F-75005, Paris, France}
\email{jean-francois.babadjian@upmc.fr}

\address[G.A. Francfort]{Universit\'e Paris-Nord, LAGA, Avenue J.-B. Cl\'ement, 93430 - Villetaneuse, France \& Institut Universitaire de France}
\email{gilles.francfort@univ-paris13.fr}

\date{\today}


\begin{abstract} This note is devoted to a rigorous derivation of rigid-plasticity as the limit of elasto-plasticity when the elasticity tends to infinity.
\end{abstract}

\maketitle

%\tableofcontents

\section{Introduction}

Small strain elasto-plasticity is formally modeled as follows. Consider a homogeneous elasto-plastic material occupying a volume $\O \subset \R^n$ with Hooke's law (elasticity tensor) $\CC$. Assume that the body is subjected to a time-dependent loading process during a time interval $[0,T]$ with, say, $f(t)$ as body loads, $g(t)$ as surface loads  on a part $\Gamma_N$ of $\partial\O$, and $w(t)$ as displacement loads (hard device) on the complementary part $\Gamma_D$ of $\partial\O$. Denoting by $Eu(t)$ the infinitesimal strain at $t$, that is, the symmetric part of the spatial gradient of the displacement field $u(t)$ at $t$, small strain elasto-plasticity requires that $Eu(t)$ decompose additively as
$$Eu(t)=e(t)+p(t) \text{ in $\O$,\; with }u(t)=w(t) \text{ on }\Gamma_D$$
where $e(t)$ is the elastic strain and $p(t)$ the plastic strain. The elastic strain is related to the stress tensor $\sigma(t)$ through the constitutive law of linearized elasticity $\sigma(t)=\CC e(t)$. In a quasi-static setting, the equilibrium equations read as
$$
\div\sigma(t)+ f(t) \quad \text{in }\O,  \quad \sigma(t)\nu = g(t) \text{ on }\Gamma_N,$$
where $\nu$ denotes the outer unit normal to $\partial\O$. In plasticity, the stresses are constrained to remain below a yield stress at which permanent strains appear. 
Specifically, the deviatoric stress $\sigma_D(t)$ must belong to a fixed compact and convex subset $K$ of the deviatoric (trace free) matrices
$$\sigma_D(t) \in K.$$
If $\sigma_D(t)$ lies inside the interior of $K$, the material behaves elastically ($p(t)=0$). On the other hand, if $\sigma_D(t)$ reaches the boundary of $K$ (called the yield surface), a plastic flow may develop, so that, after unloading, there will remain a non-trivial permanent plastic strain $p(t)$.  Its evolution is described by the so-called flow rule
$$\dot p(t) \in N_K(\sigma_D(t))$$ 
where $N_K(\sigma_D(t))$ is the normal cone to $K$ at $\sigma_D(t)$. By arguments of convex analysis, the flow rule can be equivalently written as Hill's principle of maximum plastic work
$$\sigma_D(t){\,:\,}\dot p(t)=\max_{\tau_D \in K}Ê\tau_D{\,:\,}\dot p(t)=:H(\dot p(t)),$$
where $H$ is the support function of $K$, and $H(\dot p(t))$ identifies with the plastic dissipation.

In this self-contained note, we propose to show that rigid plasticity -- that is the model where one formally sets $\CC=\infty$ (and correspondingly $\dot p(t)=E\dot u(t), \; \div \dot u(t)=0$) in the system above -- can be derived as an asymptotic limit of small strain elasto-plasticity as $\CC$ actually gets larger and larger.  Rigid-plastic models are particularly useful in order to compute analytical solutions in a plane-strain setting. Indeed, inside the plastic zone, the stress equations can  be  formally written as a non-linear hyperbolic system which is solved by the method of characteristics. The family of characteristics are the so-called {\it slip lines} along which some combinations of the stress remain constants, while the tangential velocities can jump. It thus seems appropriate to  rigorously derive rigid-plasticity in order to investigate the hyperbolic structure of the equations. However, this later task falls outside the scope of the present work.

\medskip

Notationwise, we denote by $\Ms$ the set of symmetric $n \times n$ matrices. If $A$ and $B \in \Ms$, we use the Euclidean scalar product $A:B:=\operatorname{tr}(AB)$ and the associated Euclidean norm $|A|:=\sqrt{A:A}$. The subset $\MD$ of $\Ms$ stands for trace free symmetric matrices. If $A \in \Ms$, it can be orthogonally decomposed as 
$$A=A_D+\frac{\operatorname{tr}A}{n}I,$$
where $A_D \in \MD$, and $I$ is the identity matrix in $\R^n$. The notation $\odot$ stands for the symmetrized tensor product between vectors in $\R^n$, {\it i.e.}, if $a$ and $b \in \R^n$, $(a\odot b)_{ij}=(a_ib_j+a_jb_i)/2$ for all $1 \leq i,j\leq n$. Note in particular that $\frac{1}{\sqrt 2}|a| |b|Ê\leq |a\odot b|Ê\leq |a| |b|$.

\smallskip

The Lebesgue measure in $\R^n$ and the $(n-1)$-dimensional Hausdorff measure are denoted by $\LL^n$ and $\HH^{n-1}$, respectively.
Given a locally compact set $E \subset \R^n$ and a Euclidean space $X$, we denote by $\M(E;X)$ (or simply $\M(E)$ if $X=\R$) the space of bounded Radon measures on $E$ with values in $X$, endowed with the norm $\|\mu\|_{\M(E;X)}:=|\mu|(E)$, where $|\mu|\in \M(E)$ is the variation of the measure $\mu$.  Moreover, if $\nu$ is a non-negative Radon measure over $E$, we denote by ${d\mu}/{d\nu}$ the Radon-Nikodym derivative of $\mu$ with respect to $\nu$.

 \smallskip

We use standard notation for Lebesgue and Sobolev spaces. In particular, for $1\leq p\leq \infty$, the $L^p$-norms of the various quantities are denoted by $\| \cdot\|_p$. 
If $U \subset \R^n$ is an open set, the space $BD(U)$ of functions of bounded deformation in $U$ is made of all functions $u \in L^1(U;\R^n)$ such that $Eu\in \M(U;\Ms)$, where $Eu:=(Du+Du^T)/2$ and $Du$ is the distributional derivative of $u$. We refer to \cite{T} for general properties of this space. Finally, $H(\div,U)$ stands for the Hilbert space of all $\tau \in L^2(U;\Ms)$ such that $\div \tau \in L^2(U;\R^n)$.

\section{The elasto-plastic model}

\medskip

We now consider a homogeneous elasto-plastic material with Hooke's law  given by a fourth order tensor $\CC$ satisfying the usual symmetry properties 
\begin{equation}\label{eq:C1}
\CC_{ijkl}=\CC_{jikl}=\CC_{klij}, \quad \text{Êfor all } 1 \leq i,j,k,l \leq n,
\end{equation}
and the growth and coercity assumptions
\begin{equation}\label{eq:C2}
\alpha |\xi|^2 \leq \CC \xi : \xi \leq \beta |\xi|^2, \quad \text{ for all }\xi \in \Ms,
\end{equation}
where $\alpha$ and $\beta>0$.

It occupies the domain $\O$, a bounded and connected open subset of  $\R^n$ with at least Lipschitz boundary (see Definition \ref{admiss}) and outer normal $\nu$. Its boundary $\partial \O$ is split into the union of a Dirichlet part $\Gamma_D$ which is non empty and open in the relative topology of $\partial \O$, a Neumann part $\Gamma_N:=\partial \O \setminus \overline{\Gamma_D}$, and their common relative boundary denoted by $\partial_{\lfloor_{\partial \O}} \Gamma_D$.

Standard plasticity is characterized by the fact that the deviatoric stress is constrained to stay in a fixed compact and convex subset $K \subset \MD$ of deviatoric matrices. We further assume that 
\begin{equation}\label{eq:K}
B(0,c_*) \subset K \subset B(0,c^*),
\end{equation}
where $0<c_*<c^*<\infty$, and denote by
$$\K:=\{\sigma \in L^2(\O;\Ms) : \sigma_D(x) \in K \text{ for a.e. }x \in \O\}.$$

The support function of $K$, defined for any $p \in \MD$ by
$H(p):=\sup_{\tau \in K} \tau:p$, satisfies, according to \eqref{eq:K}, 
$$c_*|p|Ê\leq H(p) \leq c^*|p|, \quad \text{ for all }p \in \Ms.$$

On the Dirichlet part $\Gamma_D$ of the boundary,  the body is subjected to a hard device, {\it i.e.}, a boundary displacement which is the trace on $\Gamma_D$ of a function $w \in AC([0,T];H^1(\O;\R^n))$. In addition, the body is subjected to two types of forces: bulk forces $f \in AC([0,T];L^n(\O;\R^n))$, and surface forces $g \in AC([0,T];L^\infty(\Gamma_N;\R^n))$, the latter acting on the Neumann part $\Gamma_N$ of the boundary. It is classical to assume a uniform safe load condition (see \cite{Suquet}) which ensures the existence of a plastically, as well as statically  admissible state of stress $\pi$  associated with the pair $(f,g)$. Specifically,  there exists $\pi \in AC([0,T];L^2(\O;\Ms))$  and some safety parameter $c>0$ such that 
$$\begin{cases}\pi_D(t,x)+B(0,c) \subset K \text{ for a.e. }x \in \O \text{ and all }t \in [0,T]\\[2mm]
\div\pi(t)+f(t)=0 \text{ in }Ê\O, \quad \pi(t)\nu=g(t) \text{ on }\Gamma_N.\end{cases}$$

Given a boundary datum $\hat w \in H^1(\O;\R^n)$, we define the space of all kinematically admissible triples as
\begin{multline*}
\A(\hat w):=\{(u,e,p) \in BD(\O) \times L^2(\O;\Ms) \times \M(\O \cup \Gamma_D;\MD) : \\
Eu=e+p \text{ in }\O,\; p=(\hat w-u)Ê\odot \nu \text{ on } \Gamma_D\},
\end{multline*}
where we still denote by $u$ the trace of $u$ on $\partial \O$ (see \cite{B}). We also define the space of all statically admissibles stresses as
$$\Sigma:=\{\sigma \in L^2(\O;\Ms) : \operatorname{div}\sigma\in L^n(\O;\R^n), \; \sigma\nu\in L^\infty(\Gamma_N;\R^n),\; \sigma_D \in L^\infty(\O;\MD)\},$$
where $\sigma\nu$ is the normal trace of $\sigma \in H(\operatorname{div},\O)$ which is well defined as an element of $H^{-1/2}(\Gamma_N;\R^n)$, the dual space of $H^{1/2}_{00}(\Gamma_N;\R^n)$.

Following \cite[Section 6]{FG}, we introduce the following class of domains  for which  a  meaningful duality pairing between stresses and strains can be defined. Note that the class contains in particular $\C^2$-domains \cite{KT}, as well as hypercubes where $\Gamma_D$ is one of its faces \cite[Section 6]{FG}.
\begin{defn}\label{admiss}
We say that $\O$ is admissible if for any $\sigma \in \Sigma$, and any $p \in \M(\O \cupÊ\Gamma_D;\MD)$, with $(u,e,p) \in \A(\hat w)$ for some $\hat w \in H^1(\O;\R^n)$, $u \in BD(\O)$ and $e \in L^2(\O;\Ms)$, the distribution defined for all $\varphi \in \C^\infty_c(\R^n)$ by
\begin{multline*}
\langle [\sigma_D:p],\varphi\rangle:=\int_\O \varphi \sigma :(E\hat w-e)\, dx - \int_\O \varphi \operatorname{div}\sigma \cdot (u-\hat w)\, dx \\
- \int_\O \sigma :[(u-\hat w)\odot \nabla \varphi]\, dx + \int_{\Gamma_N} \varphi \sigma\nu \cdot (u-\hat w)\, d\HH^{n-1}
\end{multline*}
extends to a bounded Radon measure in $\R^n$ with $|[\sigma_D:p]| \leq \|\sigma_D\|_\infty |p|$. In this case, its mass is given by
\begin{equation}\label{eq:duality}
\langle \sigma_D,p\rangle:=\langle [\sigma_D:p],1\rangle=\int_\O  \sigma :(E\hat w-e)\, dx - \int_\O  \operatorname{div}\sigma \cdot (u-\hat w)\, dx + \int_{\Gamma_N} \sigma\nu \cdot (u-\hat w)\, d\HH^{n-1}.
\end{equation}
\end{defn}

For any $e \in L^2(\O;\Ms)$,  the elastic energy is
$$\Q(e)=\frac12 \int_\O \CC e:e\, dx,$$
while, for any $p \in \M(\O \cup \Gamma_D;\MD)$,  the dissipation energy is the convex functional of measure (see \cite{GS,DT})
$$\HH(p):=\int_{\O \cup \Gamma_D} H \left(\frac{dp}{d|p|} \right) d|p|.$$
%where $\frac{dp}{d|p|} $ stands for the Radon-Nikod\'ym derivative of $p$ with respect to its variation measure $|p|$.

If  $p: [0,T] \to\M(\O \cup \Gamma_D;\MD)$, we define the total dissipation between times $a$ and $b$ by
$$\V_\HH(p;[a,b]):=\sup\left\{\sum_{i=1}^N \HH(p(t_i) -p^\e(t_{i-1})) : N \in \N,\; a=t_0<t_1<\cdots<t_N=b \right\}.$$

If additionally $p \in AC([0,T];\M(\O \cup \Gamma_D;\MD))$, then \cite[Theorem 7.1]{DMDSM} shows that
$$\V_\HH(p;[a,b])=\int_a^b \HH(\dot p(s))\, ds.$$

We finally impose the following initial condition on the evolution: $(u_0,e_0,p_0) \in \A(w(0))$ with $\sigma_0:=\CC e_0$  such that
$$\div\sigma_0+f(0)=0  \text{ in }Ê\O, \quad \sigma_0 \nu=g(0) \text{ on }\Gamma_N, \quad (\sigma_0)_D \in \K.$$ 

The following existence result has been established in \cite{DMDSM,FG}.

\begin{thm}\label{thm:DMDSM}
Under the previous assumptions, there exist a quasi-static evolution, {\it i.e.} a mapping $t \mapsto (u(t),e(t),p(t))$ with the following properties
$$u \in AC([0,T];BD(\O)),\; \sigma, \; e \in AC([0,T];L^2(\O;\Ms)),\; p \in AC([0,T]; \M(\O \cup \Gamma_D;\MD)),$$
$$(u(0),e(0),p(0))=(u_0,e_0,p_0),$$
and for all $t \in [0,T]$,
$$
\begin{cases}
Eu(t)=e(t)+p(t) \text{ in }\O, \\
p(t)=(w(t)-u(t))\odot \nu \text{ on }Ê\Gamma_D,\\
\sigma(t)=\CC e(t) \text{ in }\O,
\end{cases}
$$

$$
\begin{cases}
\div\sigma(t)+f(t)=0 \text{ in }\O, \\
\sigma(t)\nu=g(t) \text{ on }\Gamma_N,\\
 \sigma_D(t) \in \K,
 \end{cases}
$$
and for a.e. $t \in [0,T]$,
\begin{equation}\label{eq:flow-rule}
H(\dot p(t))=[\sigma_D(t): \dot p(t)] \text{ in }Ê\M(\O \cup \Gamma_D;\MD).
\end{equation}
\end{thm}

\begin{rem}\label{rem:flow-rule}
Equation \eqref{eq:flow-rule} is a measure-theoretic formulation of the usual flow rule of perfect plasticity. Using the definition \eqref{eq:duality} of duality, it can be equivalently written as an energy balance
\begin{multline*}
\Q(e(t)) + \int_0^t \HH (\dot p(s))\, ds = \Q(e_0) + \int_0^t\int_\O \sigma(s):E\dot w(s)\, dx \, ds\\
+\int_0^t\int_\O f(s)\cdot (\dot u(s)-\dot w(s))\, dx \, ds+\int_0^t\int_{\Gamma_N} g(s)\cdot (\dot u(s)-\dot w(s))\, d\HH^{n-1} \, ds,
\end{multline*}
or equivalently, according to the safe-load condition, 
\begin{multline}\label{eq:energy-eq2}
\Q(e(t)) + \int_0^t \HH (\dot p(s))\, ds-\int_0^t \langle \pi_D(s),\dot p(s)\rangle\, ds+\int_\O \pi(t): (Ew(t)-e(t))\, dx \\
= \Q(e_0)+ \int_\O \pi(0): (Ew(0)-e_0)\, dx + \int_0^t\int_\O \sigma(s):E\dot w(s)\, dx \, ds\\
+\int_0^t \int_\O \dot \pi(s):(Ew(s)-e(s))\, dx \, ds.
\end{multline}
\end{rem}

\section{The rigid-plastic model}
 
In order to derive the rigid-plastic model from elasto-plasticity, we assume that 
\begin{equation}\label{eq:hypCC}
\CC^\e=\e^{-1}\CC, \quad \text{where $\CC$ satisfies \eqref{eq:C1} and \eqref{eq:C2},}
\end{equation}
and $\e \to 0^+$. In addition, we suppose that the boundary data are compatible with rigid plasticity, that is
\begin{equation}\label{eq:divw}
\operatorname{div}w(t)=0 \text{ in }\O,
\end{equation}
and, for simplicity, that the initial data satisfy
\begin{equation}\label{eq:init-cond}
e_0=\sigma_0=0.
\end{equation}

\begin{thm}
\label{thm}
Let $u^\e$, $e^\e$, $p^\e$ and $\sigma^\e$ be the solutions given by Theorem \ref{thm:DMDSM}. There exist a subsequence (not relabeled), and functions $u \in  AC([0,T];BD(\O))$ and $\sigma \in L^2(0,T;L^2(\O;\Ms))$  such that
\begin{eqnarray*}
& u^\e(t) \wto u(t) \text{ weakly* in $BD(\O)$, for all $t\in [0,T]$},\\
& \sigma^\e \wto \sigma \text{ weakly in }L^2(0,T;L^2(\O;\Ms)).
\end{eqnarray*}
Denoting by $v:=\dot u \in L^\infty_{w*}(0,T;BD(\O))$, then for a.e. $t \in [0,T]$, we have
\begin{equation}\label{eq:rigid-plastic}
\begin{cases}
-\operatorname{div}\sigma(t)=f(t) \text{ in }Ê\O,\\
\sigma(t)\nu=g(t) \text{ on }\Gamma_N,\\
\sigma(t) \in \K,
\end{cases}
\qquad
\begin{cases}
\operatorname{div}v(t)=0 \text{ in }\O,\\
(\dot w(t)-v(t))\cdot \nu=0 \text{ on }\Gamma_D,\\
H(Ev(t))=[\sigma_D(t):Ev(t)] \text{ in }\O\cup \Gamma_D.
\end{cases}
\end{equation}
\end{thm}

The remaining of this paper is devoted to the proof of Theorem \ref{thm}.

\begin{rem}
Although $Eu(t)$ is a measure {\it a priori} defined in $\O$, we  tacitly extend it by $(w(t)-u(t))\odot \nu$ on $\Gamma_D$ so that $Eu(t) \in \M(\O \cup \Gamma_D;\MD)$.
\end{rem}

\begin{rem}
In contrast with the framework of classical elasto-plasticity, that of rigid plasticity only involves the velocity field, and not the displacement field itself. As expressed above, time is merely a parameter, although the associated measurability properties of the various fields are obtained through the limit process $\e\searrow 0$ and would be difficult to obtain directly from the limit formulation.
\end{rem}

\subsection{A priori estimates}
In this section all constants are independent of $\e$. We start with an estimate of the stress. 
Since $\sigma^\e_D(t) \in K$ in $\O$, and $K$ is bounded by \eqref{eq:K}, we first deduce that 
\begin{equation}\label{eq:est1}
\sup_{t \in [0,T]}  \|\sigma^\e_D(t)\|_\infty \leq C.
\end{equation}
The following result allows us to bound the hydrostatic stress.

\begin{lem}
There exists a bounded sequence $(c^\e)_{\e>0}$ in $L^2(0,T)$ such that for each $\e>0$,
$$\int_0^T\left\|\frac{\operatorname{tr}\sigma^\e(t)}{n}+c^\e(t)\right\|^2_2\, dt \leq C.$$
\end{lem}

\begin{proof}
Since the mapping $t \mapsto \sigma^\e(t)$ belongs to $L^2(0,T;H(\operatorname{div},\O))$, there is a sequence $(\sigma^\e_k)_{k \in \N}$ of $H(\operatorname{div},\O)$-valued simple functions such that $\sigma^\e_k \to \sigma^\e$ strongly in $L^2(0,T;H(\operatorname{div},\O))$ as $k \to +\infty$. For all $k \in \N$ and all $t \in [0,T]$, we have 
$$\nabla \left(\frac{\operatorname{tr}\sigma^\e_k(t)}{n}\right)=\div\sigma^\e_k(t)-\div(\sigma^\e_k)_D(t) \text{ in } \O$$
which leads to
$$\int_0^T \left\|\nabla \left(\frac{\operatorname{tr}\sigma^\e_k(t)}{n}\right)\right\|^2_{H^{-1}(\O;\R^n)}\, dt \leq \int_0^T \|\div\sigma^\e_k(t)\|_{H^{-1}(\O;\R^n)}^2\, dt +\int_0^T \|(\sigma^\e_k)_D(t)\|^2_2\, dt.$$
Since $\div\sigma^\e_k \to \div \sigma^\e$ in $L^2(0,T;L^2(\O;\R^n))$ and $-\div \sigma^\e=f \in L^2(0,T;L^2(\O;\R^n))$, we deduce that the first integral in the right-hand-side of the previous inequality is uniformly bounded with respect to $\e$ and $k$. The second integral is bounded as well since $(\sigma^\e_k)_D \to \sigma^\e_D$ in  $L^2(0,T;L^2(\O;\MD))$, and $(\sigma^\e_D)_{\e>0}$ is uniformy bounded in that space in view of \eqref{eq:est1}. Consequently, there exists a constant $C>0$ (independent of $k$ and $\e$) such that
$$\int_0^T \left\|\nabla \left(\frac{\operatorname{tr}\sigma^\e_k(t)}{n}\right)\right\|^2_{H^{-1}(\O;\R^n)}\, dt \leq C.$$

Next, according to  \cite[Corollary 2.1]{GR} (see also \cite[Lemma 9]{Tartar} in the case of smooth boundaries), for each $\e>0$, $kÊ\in \N$ and $t \in [0,T]$, there exists some $c_k^\e(t) \in \R$ such that
$$\left\|\frac{\operatorname{tr}\sigma_k^\e(t)}{n}+c_k^\e(t)\right\|_2 \leq C_\O \left\|\nabla \left(\frac{\operatorname{tr}\sigma_k^\e(t)}{n}\right)\right\|_{H^{-1}(\O;\R^n)},$$
for some constant $C_\O >0$ only depending on $\O$. Note that, since the mapping $t \mapsto \operatorname{tr}\sigma_k^\e(t)$ is a simple $L^2(\O)$)-valued function, $t \mapsto c_k^\e(t)$ is a simple real-valued measurable function as well. Additionally,
\begin{equation}\label{eq:1137}
\int_0^T \left\|\frac{\operatorname{tr}\sigma_k^\e(t)}{n}+c_k^\e(t)\right\|^2_2\, dt \leq C,
\end{equation}
where $C>0$ is again independent of $k$ and $\e$. Setting  $\hat \sigma_k^\e:=\sigma_k^\e+c_k^\e\ I$ yields
$$\int_0^T \|\hat \sigma_k^\e(t)\|^2_{H(\div,\O)}\, dt \leq C,$$
and thus, 
$$\int_0^T \|\hat\sigma_k^\e(t)\nu\|^2_{H^{-1/2}(\Gamma_N;\R^n)}\, dtÊ\leq C.$$ 
Using that $\sigma_k^\e\nu\to \sigma^\e\nu=g$ in $L^2(0,T;H^{-1/2}(\Gamma_N;\R^n))$ and that $g \in L^2(0,T;L^2(\Gamma_N;\R^n))$, we obtain  
\begin{multline}\label{eq:1138}
\int_0^T |c_k^\e(t)|^2\, dt \|\nu\|^2_{H^{-1/2}(\Gamma_N;\R^n)}\\
 \leq \int_0^T \|\hat\sigma_k^\e(t)\nu\|^2_{H^{-1/2}(\Gamma_N;\R^n)}\, dt +\int_0^T \|\sigma^\e_k(t)\nu\|^2_{H^{-1/2}(\Gamma_N;\R^n)}\, dt \leq C,
\end{multline}
for some constant $C>0$, independent of $k $ and $\e$. Therefore, the sequence $(c_k^\e)_{k \in \N}$ is bounded in $L^2(0,T)$ and a subsequence converges weakly in that space to some $c^\e \in L^2(0,T)$. Passing to the lower limit in \eqref{eq:1137} implies that 
$$\int_0^T \left\|\frac{\operatorname{tr}\sigma^\e(t)}{n}+c^\e(t)\right\|^2_2\, dt \leq C,$$
while \eqref{eq:1138} shows that $(c^\e)_{\e>0}$ is bounded in $L^2(0,T)$.
\end{proof}

As a consequence of the previous result and of \eqref{eq:est1}, we deduce that 
\begin{equation}\label{eq:est1bisbis}
\int_0^T\|\sigma^\e(t)\|^2_2\, dt \leq C.
\end{equation}

Next, according to the energy balance \eqref{eq:energy-eq2}, \cite[Lemma 3.2]{DMDSM}, assumptions \eqref{eq:divw}--\eqref{eq:init-cond}, and Cauchy-Schwarz inequality, we infer that
\begin{multline*}\frac{1}{2} \int_\O\CC^\e e^\e(t):e^\e(t)\, dx \leq\int_\O \pi(t):(e^\e(t)-Ew(t))\, dx+\int_\O \pi(0):Ew(0)\, dx\\
 +\int_0^t  \int_\O \sigma^\e_D(s):E\dot w(s)\, dx\, ds +\int_0^t\int_\O \dot \pi(s):(E w(s)-e^\e(s))\, dx \, ds\\ 
 \leq C \left(\sup_{t \in [0,T]} \|\pi(t)\|_2+\int_0^T \|\dot \pi(s)\|_2\, ds\right) \left(\sup_{t \in [0,T]}\|e^\e(t)\|_2 +\sup_{t \in [0,T]} \|Ew(t)\|_2 \right)\\
+\sup_{t \in [0,T]} \|\sigma^\e_D(t)\|_\infty \int_0^T \|E\dot w(s)\|_2\, ds,
\end{multline*}
which implies, according to the assumption \eqref{eq:hypCC} on $\CC^\e$ together with Young's inequality, that
\begin{equation}\label{eq:est2}
\sup_{t \in [0,T]} \|e^\e(t)\|_2 \leq C \sqrt \e.
\end{equation}

\medskip 

Using again the energy balance \eqref{eq:energy-eq2}, Cauchy-Schwarz inequality and \eqref{eq:est2}, we find that
\begin{multline*}
\int_0^t \HH (\dot p^\e(s))\, ds -\int_0^t \langle \pi_D(s),\dot p^\e(s)\rangle\, ds\leq  \int_\O \pi(t):(e^\e(t)-Ew(t))\, dx+\int_\O \pi(0):Ew(0)\, dx\\
 +\int_0^t  \int_\O \sigma^\e_D(s):E\dot w(s)\, dx\, ds +\int_0^t\int_\O \dot \pi(s):(E w(s)-e^\e(s))\, dx \, ds\leq C.
\end{multline*}
Applying \cite[Lemma 3.2]{DMDSM} again yields
\begin{equation}\label{eq:est3}
\int_0^T \|\dot p^\e(s)\|_{\M(\O\cupÊ\Gamma_D;\MD)}\, ds \leq C,
\end{equation}
and thus 
\begin{equation}\label{eq:est3bis}
\sup_{t \in [0,T]} \|p^\e(t)\|_{\M(\O\cup \Gamma_D;\MD)} \leq C.
\end{equation}
\medskip

For the displacement, Poincar\'e-Korn's inequality (see \cite[Chap. 2, Rmk. 2.5(ii)]{T}) yields
\begin{eqnarray}\label{eq:est4}
\|u^\e(t)\|_{BD(\O)} & \leq & c \left(\int_{\Gamma_D} |u^\e(t)|\, d\HH^{n-1} + \|Eu^\e(t)\|_{\M(\O;\Ms)}\right)\nonumber\\
& \leq & c \left(\int_{\Gamma_D} |w(t)|\, d\HH^{n-1}+ \int_{\Gamma_D} |u^\e(t)-w(t)|\, d\HH^{n-1}+ \|Eu^\e(t)\|_{\M(\O;\Ms)}\right)\nonumber\\
& \leq & c\left(\|w(t)\|_{L^1(\Gamma_D;\R^n)}+\|p^\e(t)\|_{\M(\O\cup\Gamma_D;\MD)}+\|e^\e(t)\|_2\right)\leq C,
\end{eqnarray}
where we have used \eqref{eq:est2} and \eqref{eq:est3bis} in the last inequality.

\subsection{Convergences}

According to the stress estimate \eqref{eq:est1bisbis}, there exist a subsequence (not relabeled) and $\sigma \in L^2(0,T;L^2(\O;\Ms))$ such that 
\begin{equation}\label{eq:conv-sigma}
\sigma^\e \wto  \sigma \text{ weakly in $L^2(0,T;L^2(\O;\Ms))$}.
\end{equation}
Consequently, since for all $t \in [0,T]$, we have $-\operatorname{div}\sigma^\e(t)=f(t)$ in $\O$ and $\sigma^\e(t)\nu=g(t)$ on $\Gamma_N$, we infer that  for a.e. $t \in [0,T]$,
$$-\operatorname{div}Ê\sigma(t)=f(t) \text{ in }\O,\quad \sigma(t) \nu=g(t) \text{ on }Ê\Gamma_N.$$
In addition, since $\sigma^\e_D(t) \in \K$ for all $t \in [0,T]$, then
$$\sigma_D(t) \in \K  \text{ for a.e. }t \in [0,T].$$

\medskip

We then apply Helly's selection principle (see \cite[Theorem 3.2]{MM}) which ensures, thanks to \eqref{eq:est3}, the existence of a further subsequence (independent of time and still not relabeled) such that
\begin{equation}\label{eq:conv-p}
p^\e(t) \wto p(t) \text{ weakly* in $\M(\O \cup \Gamma_D;\MD)$, for all $t\in [0,T]$,}
\end{equation}
for some $p \in BV([0,T];\M(\O\cup \Gamma_D;\MD))$. 

\medskip 

Next according to \eqref{eq:est2}, we have that 
\begin{equation}\label{eq:conv-e}
e^\e \to 0\text{ strongly in }L^\infty(0,T;L^2(\O;\Ms)).
\end{equation}

\medskip

Finally, as a consequence of the displacement estimate \eqref{eq:est4}, for each $t \in [0,T]$, there exists a further subsequence $(u^{\e_j}(t))_{j \in \N}$ (now possibly depending on $t$) such that $u^{\e_j}(t) \wto u(t)$ weakly* in $BD(\O)$, for some $u(t) \in BD(\O)$. Note that by \eqref{eq:conv-p}--\eqref{eq:conv-e}, for a.e. $t \in [0,T]$, one has $Eu(t)=p(t)$ in $\O$ and, by \cite[Lemma 2.1]{DMDSM}, $p(t)=(w(t)-u(t))\odot \nu$ on $\Gamma_D$ which shows that $u(t)$ is uniquely determined, and thus that the full sequence 
\begin{equation}\label{eq:conv-u}
u^\e(t) \wto u(t) \text{ weakly* in $BD(\O)$, for all $t\in [0,T]$}.
\end{equation}
In particular, since $Eu(t)=p(t) \in \M(\O \cup \Gamma_D;\MD)$, we also deduce that
\begin{equation}\label{eq:disp}
\operatorname{div} u(t)=0 \text{ in }Ê\O, \quad (w(t)-u(t))\cdot \nu=0 \text{ on }\Gamma_D.
\end{equation}

\subsection{Flow rule}

According to the energy balance \eqref{eq:energy-eq2} and the fact that the plastic strain $p^\e \in AC([0,T];\M(\overline \O;\MD))$, we can integrate by parts in time, so that for all $t \in [0,T]$,
\begin{multline*}
\V_\HH(p^\e;[0,t])+\int_\O \pi(t): (Ew(t)-e^\e(t))\, dx-\langle \pi_D(t),p^\e(t)\rangle \\
\leq \int_\O \pi(0): Ew(0)\, dx -\langle \pi_D(0),p_0\rangle+ \int_0^t\int_\O \sigma_D^\e(s):E\dot w(s)\, dx \, ds\\
+\int_0^t \int_\O \dot \pi(s):(Ew(s)-e^\e(s))\, dx \, ds-\int_0^t\langle \dot \pi_D(s),p^\e(s)\rangle\, ds.
\end{multline*}
Since by \eqref{eq:conv-p}--\eqref{eq:conv-u} $p^\e(t) \wto Eu(t)$ weakly* in $\M(\O\cup \Gamma_D;\MD)$ for a.e. $t \in [0,T]$, Reshetnyak lower semicontinuity theorem, \eqref{eq:conv-sigma}, \eqref{eq:conv-e}, \eqref{eq:conv-u} and the definition \eqref{eq:duality} of duality ensures that
\begin{multline}\label{eq:ineq1}
\V_\HH(Eu;[0,t])+\int_\O \pi(t): Ew(t)\, dx-\langle \pi_D(t),Eu(t)\rangle \\
\leq \int_\O \pi(0): Ew(0)\, dx -\langle \pi_D(0),Eu_0\rangle+ \int_0^t\int_\O \sigma_D(s):E\dot w(s)\, dx \, ds\\
+\int_0^t \int_\O \dot \pi(s):Ew(s)\, dx \, ds-\int_0^t\langle \dot \pi_D(s),Eu(s)\rangle\, ds.
\end{multline}

\medskip

We now show the converse inequality.  Since   $\sigma_D \in L^1(0,T;L^2(\O;\MD))$, while $u-w \in L^1(0,T;L^{\frac{n}{n-1}}(\O;\R^n))$, and $u-w \in L^1(0,T;L^1(\Gamma_N;\R^n))$, \cite[Lemma 7.5]{DMDSS} implies the existence of a subdivision $0=t_0<t_1<\cdots <t_k=t$ of the time interval $[0,t]$ such that
$$ \sum_{i=1}^k \chi_{[t_{i-1},t_i[} (\sigma_D(t_i), u(t_{i})-w(t_{i}), u(t_{i})-w(t_{i}))\to (\sigma_D, u-w, u-w)
$$
and
$$ \sum_{i=1}^k \chi_{[t_{i-1},t_i[} (\sigma_D(t_{i-1}), u(t_{i-1})-w(t_{i-1}), u(t_{i-1})-w(t_{i-1}))\to (\sigma_D, u-w, u-w)
$$
strongly in $L^1(0,T;L^2(\O;\MD)) \times L^1(0,T;L^{\frac{n}{n-1}}(\O;\R^n)) \times L^1(0,T;L^1(\Gamma_N;\R^n))$, as $\max_{1 \leq i \leq k}(t_{i}-t_{i-1}) \to 0$. According to Proposition 3.9 in \cite{FG} and to the fact that $\O$ is admissible, we infer that for each $1 \leq i \leq k$,
\begin{multline*}
\HH(Eu(t_i)-Eu(t_{i-1})) \geq \langle \sigma_D(t_i),Eu(t_i)-Eu(t_{i-1})\rangle\\
 = \int_\O \sigma_D(t_i):(Ew(t_i)-Ew(t_{i-1}))\, dx + \int_\O  f(t_i) \cdot (u(t_i)-u(t_{i-1})- w(t_i)+w(t_{i-1}))\, dx \\
 + \int_{\Gamma_N}g(t_i) \cdot  (u(t_i)-u(t_{i-1})- w(t_i)+w(t_{i-1}))\, d\HH^{n-1}.
 \end{multline*}
 Summing up for $i=1,\ldots,k$, and performing discrete integration by parts yields
\begin{multline*}
\V_\HH(Eu,[0,t]) \geq \sum_{i=1}^k \int_{t_{i-1}}^{t_i}\int_\O\sigma_D(t_i):E\dot w(s)\, dx \, ds\\
- \sum_{i=1}^{k-1} \int_{t_{i}}^{t_{i+1}}\int_\O \dot f(s)\cdot (u(t_i)-w(t_i))\, dx \, ds - \sum_{i=1}^{k-1} \int_{t_{i}}^{t_{i+1}} \int_{\Gamma_N} \dot g(s)\cdot (u(t_i)-w(t_i))\, d\HH^{n-1} \, ds\\
+\int_\O f(t)\cdot (u(t)-w(t))\, dx+\int_{\Gamma_N}  g(t)\cdot (u(t)-w(t))\, d\HH^{n-1}\\
-\int_\O f(t_1)\cdot (u_0-w(0))\, dx-\int_{\Gamma_N}  g(t_1)\cdot (u_0-w(0))\, d\HH^{n-1}.
 \end{multline*}
Passing to the limit as $\max_{1 \leq i \leq k}(t_{i}-t_{i-1}) \to 0$, and invoking the dominated convergence theorem yields
\begin{multline*}
\V_\HH(Eu,[0,t]) \geq \int_0^{t}\int_\O\sigma_D(s):E\dot w(s)\, dx \, ds\\
-\int_0^{t}\int_\O \dot f(s)\cdot (u(s)-w(s))\, dx \, ds - \int_0^{t}\int_{\Gamma_N} \dot g(s)\cdot (u(s)-w(s))\, d\HH^{n-1} \, ds\\
+\int_\O f(t)\cdot (u(t)-w(t))\, dx+\int_{\Gamma_N}  g(t)\cdot (u(t)-w(t))\, d\HH^{n-1}\\
-\int_\O f(0)\cdot (u_0-w(0))\, dx-\int_{\Gamma_N}  g(0)\cdot (u_0-w(0))\, d\HH^{n-1},
 \end{multline*}
 and using the definition \eqref{eq:duality} of duality
 \begin{multline*}
\V_\HH(Eu;[0,t])+\int_\O \pi(t): Ew(t)\, dx-\langle \pi_D(t),Eu(t)\rangle \\
\geq \int_\O \pi(0): Ew(0)\, dx -\langle \pi_D(0),Eu_0\rangle+ \int_0^t\int_\O \sigma_D(s):E\dot w(s)\, dx \, ds\\
+\int_0^t \int_\O \dot \pi(s):Ew(s)\, dx \, ds-\int_0^t\langle \dot \pi_D(s),Eu(s)\rangle\, ds.
\end{multline*}
Thus, combining with \eqref{eq:ineq1} leads to the equality in the previous inequality,
or still, integrating by parts with respect to time
\begin{multline}\label{eq:ee}
\V_\HH(Eu;[0,t]) = \langle \pi_D(t),Eu(t)\rangle-\langle \pi_D(0),Eu_0\rangle\\
+\int_0^t \int_\O (\sigma_D(s)-\pi_D(s)):E\dot w(s)\, dx \, ds-\int_0^t\langle \dot \pi_D(s),Eu(s)\rangle\, ds.
 \end{multline}

According to \cite[Lemma 3.2]{DMDSM}, for all $0 \leq t_1 \leq t_2 \leq T$,  
\begin{eqnarray*}
c\|Eu(t_2)-Eu(t_1)\|_{\M(\O \cup \Gamma_D;\MD)} & \leq & \HH(Eu(t_2)-Eu(t_1))-\langle \pi_D(t_2),Eu(t_2)-Eu(t_1)\rangle \\
& \leq & \V_\HH(Eu,[t_1,t_2])-\langle \pi_D(t_2),Eu(t_2)-Eu(t_1)\rangle .
\end{eqnarray*}
In view of \eqref{eq:ee}, we get that
\begin{multline*}
c\|Eu(t_2)-Eu(t_1)\|_{\M(\O \cup \Gamma_D;\MD)} \leq\langle \pi_D(t_2)-\pi_D(t_1),Eu(t_1)\rangle\\
+\int_{t_1}^{t_2} \int_\O (\sigma_D(s)-\pi_D(s)):E\dot w(s)\, dx \, ds-\int_{t_1}^{t_2}\langle \dot \pi_D(s),Eu(s)\rangle\, ds.
\end{multline*}
Since $Eu=p$ and $p \in BV([0,T];\M(\O \cup \Gamma_D;\MD))$, we get that
$Eu \in L^\infty_{w*}(0,T;\M(\O\cup \Gamma_D;\MD))$, and thus
\begin{multline*}
c\|Eu(t_2)-Eu(t_1)\|_{\M(\O \cup \Gamma_D;\MD)} \leq \int_{t_1}^{t_2}  \Big\{\|Eu(t_1)\|_{\M(\O \cup \Gamma_D;\MD)} \|\dot \pi_D(s)\|_\infty\\
 +(\|\pi_D(s)\|_2 + \|\sigma_D(s)\|_2)\|E\dot w(s)\|_2 + \|\dot \pi_D(s)\|_\infty \|Eu(s)\|_{\M(\O \cup \Gamma_D;\MD)} \Big\}\, ds.
 \end{multline*}
The integrand being sommable, it ensures that the strain $Eu \in AC([0,T];\M(\O\cup \Gamma_D;\MD))$, and by the Poincar\'e-Korn inequality that $u \in AC([0,T];BD(\O))$. Thus, integrating by part with respect to time and space in the energy equality \eqref{eq:ee},
\begin{multline*}\int_0^t \HH(E\dot u(s))\, ds= \V_\HH(Eu,[0,t]) =  \int_0^{t}\int_\O\sigma_D(s):E\dot w(s)\, dx \, ds\\
+\int_0^{t}\int_\O  f(s)\cdot (\dot u(s)-\dot w(s))\, dx \, ds + \int_0^{t}\int_{\Gamma_N}  g(s)\cdot (\dot u(s)-\dot w(s))\, d\HH^{n-1} \, ds,
\end{multline*}
and deriving this equality with respect to time yields, thanks to \eqref{eq:duality}, for a.e. $t \in [0,T]$,
$$\HH(E\dot u(t))=\langle\sigma_D(t),E\dot u(t)\rangle.$$
Since, by \cite[Proposition 3.9]{FG},  $H(E\dot u(t))Ê\geq [\sigma_D(t):E\dot u(t)]$ in $\M(\O\cupÊ\Gamma_D)$, we finally deduce that $H(E\dot u(t))Ê= [\sigma_D(t):E\dot u(t)]$ in $\M(\O\cup \Gamma_D)$. 

Denoting by $v=\dot u$ the velocity, we proved that $v \in L^\infty_{w*}(0,T;BD(\O))$, and recalling \eqref{eq:disp}, we have for a.e. $t \in [0,T]$,
$$\operatorname{div}v(t)=0 \text{ in }\O, \quad (\dot w(t)-v(t))\cdot \nu=0 \text{ on }\Gamma_D,$$
and
$$H(Ev(t))=[\sigma_D(t):Ev(t)] \text{ in }\O\cupÊ\Gamma_D.$$

\section{Uniqueness and regularity issues for the stress with a Von Mises yield criterion}

We now specialize to the case where  $K:=\{\tau_D \in \MD: |\tau_D|Ê\leq 1\}$. In such a setting, it is known (see \cite{BF})  when elasto-plasticity is considered the stress field is unique and belongs to $H^1_{\rm loc}(\O;\Ms)$. These properties fail in the case of rigid-plasticity as demonstrated below.
\begin{example}
Let us  consider a two-dimensional body occupying the square $\O=(0,1)^2$ in its reference configuration (the generalization to the $n$-dimensional case is obvious).  We also assume that the boundary conditions are of pure Dirichlet type with  a rigid body motion $\dot w(x)=Ax+b$  (where $A \in \mathbb M^{n \times n}$ is such that $A^T=-A^T$, and $b \in \R^n$) as boundary datum.

Then, defining $v(x)=Ax+b$ for all $x \in \Omega$ ensures that $Ev=0$ in $\O$. In particular, all equations on $v$ are satisfied. Now define the stress as 
$$\sigma(x)=\begin{pmatrix}
f(x_2) & c\\
c & g(x_1)
\end{pmatrix}$$
where $c \in \R$, $f$, $g \in L^\infty(0,1)$ so that $\operatorname{div}\sigma=0$ in $\O$. In particular
$$\sigma_D(x)=\begin{pmatrix}
\frac{f(x_2)-g(x_1)}{2} & c\\
c & \frac{g(x_1)-f(x_2)}{2}
\end{pmatrix}$$
and $|\sigma_D(x)|^2\leq 2c^2 + |f(x_2)|^2+|g(x_1)|^2$ for a.e. $x \in \O$. Assuming that $\sqrt{2c^2+\|f\|_\infty^2+\|g\|_\infty^2}<1/2$, we deduce that the one parameter family $\sigma^\lambda:=\lambda \sigma$ still satisfies $\operatorname{div}\sigma^\lambda=0$ and $|\sigma^\lambda_D|<1$ in $\O$ provided that $|\lambda| \leq 2$.
\end{example}

In general, a certain amount of uniqueness holds true as shown below. It uses a notion of precise representative for the stress field first introduced in \cite{A} (see also \cite{DMDSM}).

\begin{prop}
Let $(\sigma^1,v^1)$, $(\sigma^2,v^2) \in L^2(\O;\Ms) \times BD(\O)$ be two solutions of the rigid-plastic model \eqref{eq:rigid-plastic} at a given time $t=t_0$. Then,
\begin{itemize}
\item There exist two $|Ev^1|$-measurable functions $\hat \sigma^1_D$ and $\hat \sigma^2_D \in L^\infty_{|Ev^1|}(\O \cup \Gamma_D;\MD)$ such that
$\hat \sigma^1_D=\sigma^1$ and  $\hat \sigma^2_D=\sigma^2_D$ $\LL^n$-a.e. in $\O\cup\Gamma_D$,
and
$$\hat \sigma^1_D=\hat \sigma^2_D \quad |Ev^1|\text{-a.e. in }\O\cup\Gamma_D;$$
\item There exist two $|Ev^2|$-measurable functions $\tilde \sigma^1_D$ and $\tilde \sigma^2_D \in L^\infty_{|Ev^2|}(\O \cup \Gamma_D;\MD)$ such that
$\tilde \sigma^1_D=\sigma^1$ and $\tilde \sigma^2_D=\sigma^2_D$ $\LL^n$-a.e. in $\O \cup \Gamma_D$, and
$$\tilde \sigma^1_D=\tilde \sigma^2_D \quad |Ev^2|\text{-a.e. in }\O\cup\Gamma_D.$$
\end{itemize}
\end{prop}

\begin{proof}
Since  $(\sigma^1,v^1)$, $(\sigma^2,v^2)$ are two solutions of the rigid-plastic model \eqref{eq:rigid-plastic}, the following inequalities in $\M(\O \cup \Gamma_D)$ hold true
$$[\sigma^1_D:Ev^1]  = |Ev^1|Ê\geq [\sigma^2_D:Ev^1],\quad 
[\sigma^2_D:Ev^2]= |Ev^2|Ê\geq [\sigma^1_D:Ev^2].$$
As a consequence,
$$[(\sigma^1_D-\sigma^2_D):Ev^1] \geq 0,\quad [(\sigma^2_D-\sigma^1_D):Ev^2] \geq 0,$$
and thus, 
$$[(\sigma^1_D-\sigma^2_D):(Ev^1-Ev^2)] \geq 0.$$
In addition, by definition \eqref{eq:duality} of the duality pairing, the total mass of the measure on the left-hand side of the previous inequality is given by
$$\langle \sigma^1_D-\sigma^2_D,Ev^1-Ev^2\rangle=0.$$
It thus follows that
$$[(\sigma^1_D-\sigma^2_D):Ev^1] = 0,\quad [(\sigma^2_D-\sigma^1_D):Ev^2] = 0,$$
or still that
\begin{equation}\label{eq:mtfl}
[\sigma^1_D:Ev^1]  = |Ev^1|Ê= [\sigma^2_D:Ev^1],\quad 
[\sigma^2_D:Ev^2]= |Ev^2|Ê= [\sigma^1_D:Ev^2].
\end{equation}

Arguing as in \cite{DMDSM}, since $\LL^n$ and $E^sv^1$ are mutually singular Borel measures, it is possible to find two disjoint Borel sets $A$ and $B \subset \O \cup \Gamma_D$ such that $A \cup B=\O \cup \Gamma_D$, and $\LL^n(B)=|E^s v^1|(A)=0$. Then, defining (for $i=1,2$)
$$\hat\sigma^i_D:=\left\{\begin{array}{cl}
\sigma^i_D Ê& \LL^n\text{-a.e. in }A,\\[0.2cm]
\displaystyle \frac{dEv^1}{d|Ev^1|}& |E^s v^1|\text{-a.e. in }B,
\end{array}\right.$$
it follows that
$\hat \sigma^1_D$ and $\hat \sigma^2_D \in L^\infty_{|Ev^1|}(\O \cup \Gamma_D;\MD)$, and
$$\hat  \sigma^1_D:\frac{dEv^1}{d|Ev^1|}|Ev^1|=[\sigma^1_D:Ev^1]=|Ev^1|=[\sigma^2_D:Ev^1]= \hat \sigma^2_D:\frac{dEv^1}{d|Ev^1|}|Ev^1|.$$
By definition, we have that $\hat  \sigma^1_D=\hat  \sigma^2_D$ $|E^s v^1|$-a.e. in $\O \cup \Gamma_D$. 
In addition, taking the absolutely continuous part in \eqref{eq:mtfl} yields  (see \cite{DMDSM,FG}),
 $$ \sigma^1_D: E^av^1Ê=[\sigma^1_D:Ev^1]^a=|E^av^1|=[\sigma^2_D:Ev^1]^a= \sigma^2_D:E^a v^1.$$
Thus $\sigma^1_D=\sigma^2_D$ $\LL^n$-a.e. in $\{|E^av^1|>0\}$ and finally $\hat \sigma^1_D=\hat \sigma^2_D$ $|Ev^1|$-a.e. in $\O\cupÊ\Gamma_D$ as requested. 
\end{proof}

\section*{Acknowledgements} 

The authors wish to thank the hospitality of the Courant Institute of Mathematical Sciences at New York University where a large part of this work has been carried out.


\begin{thebibliography}{99}

%\bibitem{AFP}
%{\sc L. Ambrosio, N. Fusco, D. Pallara}: {\it Functions of bounded variation and free discontinuity problems}, Oxford Mathematical Monographs. The Clarendon Press, Oxford University Press, New York (2000).

\bibitem{A}
{\sc  G.  Anzellotti}: On  the  extremal  stress  and  displacement  in  Hencky  plasticity, 
{\it Duke Math. J.}  {\bf 51}(1) (1984) 133--147.

\bibitem{B}
{\sc J.-F. Babadjian}: Traces of functions of bounded deformation, {\it Indiana Univ. Math. J.} {\bf 64} (2015) 1271--1290. 

\bibitem{BF}
{\sc A. Bensoussan, J. Frehse}: Asymptotic behaviour of the time dependent Norton-Hoff law in plasticity theory and $H^1$ regularity, 
{\it Comment. Math. Univ. Carolin.} {\bf 37} (1996) 285--304. 

\bibitem{DMDSM} 
{\sc G. Dal Maso, A. De Simone, M. G. Mora}: Quasistatic evolution problems for linearly elastic--perfectly plastic materials, {\it Arch. Rational Mech. Anal.} {\bf 180} (2006) 237--291.

\bibitem{DMDSS}
{\sc G. Dal Maso, A. De Simone, F. Solombrino}: Quasistatic evolution for Cam-Clay plasticity: a weak formulation via viscoplastic regularization and time rescaling, {\it Calc. Var. Partial Diff. Eq.} {\bf 40} (2011) 125--181.

\bibitem{DT}
{\sc F. Demengel, R. Temam}: Convex function of a measure, {\it Indiana Univ. Math. J.} {\bf 33} (1984) 673--709.


\bibitem{FG}
{\sc G.A. Francfort, A. Giacomini}: Small strain heterogeneous elasto-plasticity revisited, {\it Comm. Pure Appl. Math.} {\bf 65} (2012) 1185--1241.

\bibitem{GR}
{\sc V. Girault, P.-A. Raviart}: {\it Finite element method for Navier-Stokes equations. Theory and algorithm}, Springer-Verlag (1986).

\bibitem{GS}
{\sc C. Goffman, J. Serrin}: Sublinear functions of measures and variational integrals, {\it Duke Math. J.} {\bf 31} (1964) 159--178.

\bibitem{KT}
{\sc R.V.  Kohn, R. Temam}: Dual spaces of stresses and strains, with applications to Hencky plasticity, {\it Appl. Math. Optim.} {\bf10} (1983) 1--35.

%\bibitem{Lub}
%{\sc J. Lubliner}: {\it Plasticity Theory\/}, Macmillan Publishing Company, New York, 1990.

\bibitem{MM}
{\sc A. Mainik, A. Mielke}:  Existence results for energetic models for rate-independent systems, {\it Calc. Var. PDEs\/} {\bf 22} (2005) 73--99.

%\bibitem{S}
%{\sc J. Salen\c con}: {\it Application of the theory of plasticity to soil mechanics}, Wiley series in geotechnical engineering (1977).

\bibitem{Suquet}
{\sc P. Suquet}: Sur les \'equations de la plasticit\'e: existence et r\'egularit\'e des solutions, {\it J. M\'ecanique} {\bf 20} (1981) 3--39.

\bibitem{Tartar}
{\sc L. Tartar}: {\it Topics in nonlinear analysis}, publications math\'ematiques d'Orsay (1982).

\bibitem{T}
{\sc R. Temam}: {\it Mathematical problems in plasticity}, Gauthier-Villars, Paris (1985). Translation of {\it Probl\`emes math\'ematiques en plasticit\'e}, Gauthier-Villars, Paris (1983).

\end{thebibliography}
\end{document}